\documentclass[reqno]{amsart}
\usepackage{hyperref}
\usepackage{graphicx}
\usepackage[latin1]{inputenc}
\usepackage[english,activeacute]{babel}

\numberwithin{equation}{section}
\newtheorem{theorem}{Theorem}[section]
\newtheorem{lemma}[theorem]{Lemma}

\newtheorem{proposition}[theorem]{Proposition}

\title[Fractional logarithmic Schr\"{o}dinger equation]
{Existence and stability of standing waves for nonlinear fractional  Schr\"{o}dinger equation with logarithmic nonlinearity}
\author[Alex H. Ardila]{}
\email{alexha@ime.usp.br}

\thanks{The  author was supported by CNPq-Brazil}

\subjclass[2010]{35Q55; 35Q40}
\keywords{Fractional logarithmic Schr\"{o}dinger equation; standing waves;  stability}

\begin{document}

\maketitle


\centerline{\scshape Alex H. Ardila}
{\footnotesize
 \centerline{Department of Mathematics, IME-USP, Universidade de S\~ao Paulo}
\centerline{Cidade Universit\'aria, CEP 05508-090, S\~ao Paulo, SP, Brazil}
} 

\begin{abstract}
In this paper we consider the nonlinear fractional logarithmic Schr\"{o}dinger equation.  By using a compactness method, we construct a unique global solution of the associated Cauchy problem in a suitable functional framework.  We also prove the existence of ground states as minimizers of the action on the Nehari manifold. Finally, we prove that the set of minimizers is a stable set for the initial value problem, that is, a solution whose initial data is near the set will remain near it for all time.
\end{abstract}

\section{Introduction}
\label{S:0}
This paper is concerned with the fractional nonlinear Schr\"{o}dinger equation with logarithmic nonlinearity
\begin{equation}
\label{0NL}
  i\partial_{t}u-(-\Delta)^{s} u+\,u\,\mathrm{Log}\,\left|u\right|^{2}=0,
\end{equation}
where $0<s<1$ and $u=u(x,t)$ is a complex-valued function of $(x,t)\in \mathbb{R}^{N}\times\mathbb{R}$, $N\geq2$. The fractional Laplacian 
$(-\Delta)^{s}$ is defined via Fourier transform as
\begin{equation}\label{Flapla}
\mathcal{F}\left[(-\Delta)^{s}u\right](\xi)=|\xi|^{2s}\mathcal{F}u(\xi),
\end{equation}
where the Fourier transform is given  by 
\begin{equation}\label{TRFou}
\mathcal{F}u(\xi)=\frac{1}{(2\pi)^{N/2}}\int_{ \mathbb{R}^{N}}u(x)e^{-i\xi\cdot x}dx.
\end{equation}
The fractional Laplacian $(-\Delta)^{s}$ is a self-adjoint operator on $L^{2}(\mathbb{R}^{N})$ with quadratic form domain $H^{s}(\mathbb{R}^{N})$ and operator domain $H^{2s}(\mathbb{R}^{N})$. Moreover,  the following spectral properties of  $(-\Delta)^{s}$ are known: $\sigma_{\rm ess}((-\Delta)^{s})=[0,\infty)$ and  $\sigma_{\rm p}((-\Delta)^{s})=\emptyset$ (see, e.g., \cite[Example 3.3]{JLFHVB}).
The nonlocal  operator $(-\Delta)^{s}$  can be seen as the infinitesimal generators of Lévy stable diffusion processes (see \cite{DAP}). Fractional powers of the Laplacian arise in a numerous variety of equations in mathematical physics and related fields; see, e.g., \cite{DAP, DAPPL,XGMX} and references therein.  Recently, a great attention has been focused on the study of problems involving the fractional Laplacian from a pure mathematical point of view. Concerning the fractional Schr\"{o}dinger  equations, let us mention \cite{YCHHG, XCYS,BGHJ, BGAZH, BGDHL,PDAMS,SB1, CHHHHW, Haee333, HH123, DWU}. 

The present paper is devoted to the analysis of existence and stability of standing waves of NLS \eqref{0NL}. If the fractional Laplacian in \eqref{0NL} is replaced by a standard Laplacian, this problem is well-known and described in detail in \cite{AHA1, CAS, CL,CALO,PHST,PHBJ}.
In this case, one can show that that there exists a unique (up to translations and phase shifts) ground state and it is orbitally stable.
The classical logarithmic NLS equation was proposed by Bialynicki-Birula and Mycielski \cite{CAS} in 1976 as a model of nonlinear wave mechanics and  has  important applications in  quantum mechanics, quantum optics, nuclear physics, open quantum systems and  Bose-Einstein condensation (see e.g. \cite{HE123, APLES} and the references therein).

The energy functional $E$ associated with problem \eqref{0NL} is 
\begin{equation}\label{EEEE}
E(u)=\frac{1}{2}\int_{\mathbb{R}^{N}}|(-\Delta)^{s/2} u|^{2}dx-\frac{1}{2}\int_{\mathbb{R}^{N}}\left|u\right|^{2}\mbox{Log}\left|u\right|^{2}dx.
\end{equation}
Unfortunately, due to the singularity of the logarithm at the origin, the functional fails to be finite as well of class $C^{1}$ on $H^{s}(\mathbb{R}^{N})$. Due  to  this  loss  of  smoothness,  it is convenient  to work in a suitable Banach space endowed with a Luxemburg type norm in order to make functional $E$  well defined and $C^{1}$ smooth. This space allows to control the singularity of the logarithmic
nonlinearity at infinity and at the origin. Indeed, we consider the reflexive Banach space (see Section \ref{S:1})
\begin{equation}\label{ASE}
W^{s}(\mathbb{R}^{N})=\left\{u\in H^{s}(\mathbb{R}^{N}):\left|u\right|^{2}\mathrm{Log}\left|u\right|^{2}\in L^{1}(\mathbb{R}^{N})\right\},
\end{equation}
then the energy functional $E$ is well-defined and of class $C^1$ on $W^{s}({\mathbb R}^{N})$.  Moreover, from Lemma \ref{APEX23},  we have that the operator $ u\rightarrow (-\Delta)^{s}u-u\,  \mathrm{Log}\left|u\right|^{2}$ is continuous from  $W^{s}(\mathbb{R}^{N})$  to $W^{-s}(\mathbb{R}^{N})$. Here, ${W}^{-s}(\mathbb{R}^{N})$ is the dual space of ${W}^{s}(\mathbb{R}^{N})$. Therefore,  if $u\in C(\mathbb{R}, W^{s}(\mathbb{R}^{N}))\cap C^{1}(\mathbb{R}, W^{-s}(\mathbb{R}^{N}))$, then equation \eqref{0NL} makes sense in $W^{-s}(\mathbb{R}^{N})$.

The next proposition gives a result on the existence of weak solutions to \eqref{0NL} in the energy space $W^{s}({\mathbb R}^{N})$. The proof is contained in Section \ref{S:2}.

\begin{proposition} \label{PCS}
For any $u_{0}\in {W}^{s}({\mathbb{R}}^{N})$, there is a unique  solution  $u\in C(\mathbb{R},{W}^{s}({\mathbb{R}^{N}}))\cap C^{1}(\mathbb{R}, {W}^{-s}(\mathbb{R}^{N}))$  of \eqref{0NL}  such that $u(0)=u_{0}$ and $\sup_{t\in \mathbb{R}}\left\|u(t)\right\|_{{W}^{s}({\mathbb{R}^{N}})}<\infty$. Furthermore, the conservation of energy and charge holds; that is, 
\begin{equation*}
E(u(t))=E(u_{0})\quad  and \quad \left\|u(t)\right\|^{2}_{L^{2}}=\left\|u_{0}\right\|^{2}_{L^{2}}\quad  \text{for all $t\in \mathbb{R}$}.
\end{equation*}
\end{proposition}
In this paper we study the existence and stability of standing waves solutions of \eqref{0NL} of the form $u(x,t)=e^{i\omega t}\varphi(x)$, 
where $\omega\in \mathbb{R}$ and $\varphi\in {W}^{s}({{\mathbb{R}^{N}}})$ is a complex valued function which has to solve the following
stationary problem
\begin{equation}\label{EP}
(-\Delta)^{s} \varphi+\omega \varphi-\varphi\, \mathrm{Log}\left|\varphi \right|^{2}=0, \quad x\in {\mathbb{R}}^{N}.
\end{equation}
For $\omega\in \mathbb{R}$, we define the following functionals of class $C^{1}$ on $W^{s}(\mathbb{R}^{N})$:
\begin{align*}
 S_{\omega}(u)&=\frac{1}{2}\int_{\mathbb{R}^{N}}|(-\Delta)^{s/2} u|^{2}dx+\frac{\omega+1}{2}\int_{\mathbb{R}^{N}}\left|u\right|^{2}dx-\frac{1}{2}\int_{\mathbb{R}^{N}}\left|u\right|^{2}\mbox{Log}\left|u\right|^{2}dx,\\
  I_{\omega}(u)&=\int_{\mathbb{R}^{N}}|(-\Delta)^{s/2} u|^{2}dx+\omega\int_{\mathbb{R}^{N}}\left|u\right|^{2}dx-\int_{\mathbb{R}^{N}}\left|u\right|^{2}\mbox{Log}\left|u\right|^{2}dx.
\end{align*}
Note that \eqref{EP} is equivalent to $S^{\prime}_{\omega}(\varphi)=0$, and $I_{\omega}(u)=\left\langle S_{\omega}^{\prime}(u),u\right\rangle$ is the so-called Nehari functional. It was shown in \cite{PDAMS} that the problem \eqref{EP} admits a sequence of weak solutions  $u_{n}\in H^{s}(\mathbb{R}^{N})$ with $S_{\omega}(u_{n})\rightarrow +\infty$ as $n\rightarrow +\infty$.

From the physical point viewpoint, an important role is played by the ground state solution of \eqref{EP}. We recall that a solution $\varphi\in {W}^{s}({\mathbb{R}^{N}})$ of \eqref{EP} is termed as a ground state if it has some minimal action among all solutions of \eqref{EP}. To be more specific,  we consider the minimization problem
\begin{align}
\begin{split}\label{MPE}
d(\omega)&={\inf}\left\{S_{\omega}(u):\, u\in {W}^{s}({\mathbb{R}^{N}})  \setminus  \left\{0 \right\},  I_{\omega}(u)=0\right\} \\ 
&=\frac{1}{2}\,{\inf}\left\{\left\|u\right\|_{L^{2}({\mathbb{R}^{N}})}^{2}:u\in  {W}^{s}({\mathbb{R}^{N}}) \setminus \left\{0 \right\},  I_{\omega}(u)= 0 \right\}, 
\end{split}
\end{align}
and define the set of ground states  by
\begin{equation*}
 \mathcal{G}_{\omega}=\bigl\{ \varphi\in {W}^{s}({\mathbb{R}^{N}}) \setminus  \left\{0 \right\}: S_{\omega}(\varphi)=d(\omega), \quad I_{\omega}(\varphi)=0\bigl\}.
\end{equation*}
The set $\bigl\{u\in{W}^{s}({\mathbb{R}^{N}}) \setminus  \left\{0 \right\},  I_{\omega}(u)=0\bigl\}$ is called the Nehari manifold. Notice that  the above set contains all stationary points of $S_{\omega}$.

The existence of minimizers for \eqref{MPE} will be obtained as a consequence of the stronger statement that any minimizing sequence for \eqref{MPE} is, up to translation, precompact in $W^{s}(\mathbb{R}^{N})$. We will show the following theorem in  Section \ref{S:3}.
\begin{theorem} \label{ESSW}
Let $N\geq 2$, $\omega\in \mathbb{R}$ and $0<s<1$. Let $\left\{ u_{n}\right\}\subseteq W^{s}(\mathbb{R}^{N})$ be a minimizing sequence for $d(\omega)$. Then there exists a family $(y_{n})\subset \mathbb{R}^{N}$ such that $\left\{u_{n}(\cdot-y_{n})\right\}$ contents a convergent subsequence in $W^{s}(\mathbb{R}^{N})$. In particular, this implies that $\mathcal{G}_{\omega}$ is not empty set for any $\omega\in \mathbb{R}$.
\end{theorem}

We remark that for any $u\in \mathcal{G}_{\omega}$,  there exists a Lagrange multiplier  $\Lambda\in \mathbb{R}$ such that $S^{\prime}_{\omega}(u)=\Lambda I^{\prime}_{\omega}(u)$. Thus, we have $\left\langle S^{\prime}_{\omega}( u),u\right\rangle=\Lambda\left\langle  I^{\prime}_{\omega}(u),u\right\rangle$. The fact that   $\left\langle S^{\prime}_{\omega}( u),u\right\rangle =I_{\omega}(u)=0$ and $\left\langle  I^{\prime}_{\omega}(u),u\right\rangle= -2\left\|u_{}\right\|_{L^{2}(\mathbb{R}^{N})}^{2}<0$, implies $\Lambda=0$; that is, $S_{\omega}^{\prime}(u)=0$. Therefore,  $u$ satisfies \eqref{EP}.

Now we are ready to state our  main result, which is a direct consequence of the result of relative compactness.

\begin{theorem} \label{2ESSW}
Let $N\geq 2$, $\omega\in \mathbb{R}$ and $0<s<1$. Then  the set $G_{\omega}$ is $W^{s}(\mathbb{R}^{N})$-stable with respect to NLS \eqref{0NL}; that is, for arbitrary $\epsilon>0$, there exist $\delta>0$ such that for any $u_{0}\in W^{s}(\mathbb{R}^{N})$, if 
\begin{equation*}
\inf_{\psi\in \mathcal{G}_{\omega}} \|u_{0}-\psi\|_{W^{s}(\mathbb{R}^{N})}<\delta,
\end{equation*}
then the solution $u(x,t)$ of the Cauchy problem \eqref{0NL} with the initial data $u_{0}$ satisfies 
\begin{equation*}
\inf_{\psi\in \mathcal{G}_{\omega}} \|u(\cdot,t)-\psi\|_{W^{s}(\mathbb{R}^{N})}<\epsilon, \quad \text{for all $t\geq 0$.}
\end{equation*}
\end{theorem}

The rest of the paper is organized as follows.  In Section \ref{S:1}, we analyse the structure of the energy space $W^{s}(\mathbb{R}^{N})$. Moreover, we recall several known results, which will be needed later. In Section \ref{S:2}, we give an idea of the proof of  Proposition \ref{PCS}.  In Section \ref{S:3} we prove, by variational techniques, the existence of a minimizer for $d(\omega)$. The stability result is proved in Section  \ref{S:4}. In the Appendix we list some properties of the Orlicz space  associated with $W^{s}(\mathbb{R}^{N})$.

{\bf Notation.} The space $L^{2}(\mathbb{R}^{N},\mathbb{C})$  will be denoted  by $L^{2}(\mathbb{R}^{N})$ and its norm by $\|\cdot\|_{L^{2}}$.  
$\left\langle \cdot , \cdot \right\rangle$ is the duality pairing between $X^{\prime}$ and $X$, where $X$ is a Banach space and $X^{\prime}$ is its dual. Finally, $2^{\ast}_{s}:=2N/(N-2s)$ for $N\geq2$ and $0<s<1$. Throughout this paper, the letter C will denote positive
constants whose value may change from line to line.

\section{Preliminaries and functional setting}  
\label{S:1}
For the sake of self-containedness, we first provide some basic properties of the fractional Sobolev spaces $H^{s}(\mathbb{R}^{N})$, which will be needed later. Consider the fractional order Sobolev space
\begin{equation*}
H^{s}(\mathbb{R}^{N})=\left\{u\in L^{2}(\mathbb{R}^{N}): \int_{\mathbb{R}^{N}} (1+|\xi|^{2s})|\hat{u}(\xi)|^{2}d\xi<\infty\right\},
\end{equation*}
where $\hat{u}=\mathcal{F}(u)$. The norm is defined by $\|u\|^{2}_{H^{s}(\mathbb{R}^{N})}=\int_{\mathbb{R}^{N}} (|\hat{u}(\xi)|^{2}+|\xi|^{2s}|\hat{u}(\xi)|^{2})d\xi$.
Notice that, by Plancherel's theorem we have  
\begin{align*}
\|u\|^{2}_{H^{s}(\mathbb{R}^{N})}&=\int_{\mathbb{R}^{N}} |\hat{u}(\xi)|^{2}d\xi+\int_{\mathbb{R}^{N}}|\xi|^{2s}|\hat{u}(\xi)|^{2}d\xi\\
&=\int_{\mathbb{R}^{N}} |{u}(x)|^{2}dx+\int_{\mathbb{R}^{N}}|\widehat{(-\Delta)^\frac{s}{2}u}(\xi)|^{2}d\xi\\
&=\int_{\mathbb{R}^{N}} |{u}(x)|^{2}dx+\int_{\mathbb{R}^{N}}|(-\Delta)^\frac{s}{2} u(x)|^{2}dx<\infty,
\end{align*}
for every $u\in H^{s}(\mathbb{R}^{N})$.  Moreover, the space $H^{s}(\mathbb{R}^{N})$ is continuously embedded into $L^{q}(\mathbb{R}^{N})$ for any $q\in[2,2^{*}_{s}]$ and compactly embedded into $L_{\text{loc}}^{q}(\mathbb{R}^{N})$ for any $q\in[2,2^{*}_{s})$, {where}  $2^{*}_{s}={2N}/{N-2s}$. See \cite{DNGPEV} for more details.

Let $0<s<1$ and $\Omega$ a smooth bounded domain of $\mathbb{R}^{N}$, we define $H^{s}(\Omega)$ as follows
\begin{equation*}
H^{s}(\Omega)=\left\{u\in L^{2}(\Omega): \frac{|u(x)-u(y)|}{|x-y|^{\frac{N}{2}+s}}\in L^{2}(\Omega\times \Omega)\right\},
\end{equation*}
endowed with the norm
\begin{equation}\label{NM2}
\|u\|^{2}_{H^{s}(\Omega)}=\int_{\Omega}|u|^{2}dx+\int_{\Omega}\int_{\Omega}\frac{|u(x)-u(y)|^{2}}{{|x-y|^{{N}+2s}}}dxdy.
\end{equation}
Also, we denote by $H_{0}^{s}(\Omega)$ the Hilbert space defined as the closure of $C^{\infty}_{0}(\Omega)$ under the norm $\|\cdot\|^{2}_{H^{s}(\Omega)}$ defined in \eqref{NM2}. The dual space $H^{-s}(\Omega)$ of $H_{0}^{s}(\Omega)$ is defined in the standard way. For a general reference on analytical properties of fractional Sobolev spaces, see \cite{DNGPEV}.

The next result is an adaptation of a classical lemma of Lions. For a proof we refer to \cite[Lemma 2.8]{SB1}.
\begin{lemma} \label{CLLQ} Let $N\geq2$ and $2^{\ast}_{s}=2N/(N-2s)$. If $\left\{u_{n}\right\}$ is bounded in $ H^{s}(\mathbb{R}^{N})$ and for some $R>0$  we have
\begin{equation*}
\lim_{n\rightarrow\infty}\sup_{y\in\mathbb{R}^{N}} \int _{B_{R}(y)}\left|u_{n}(x)\right|^{2}dx=0,
\end{equation*}
then one has $u_{n}\rightarrow 0$ in $L^{r}(\mathbb{R}^{N})$ for any $2<r<2^{\ast}_{s}$.
\end{lemma}

Now we need to introduce some notation. Define 
\begin{equation*}
F(z)=\left|z\right|^{2}\mbox{Log}\left|z\right|^{2}\quad  \text{for every  $z\in\mathbb{C}$},
\end{equation*}
and as in \cite{CL},  we define the functions  $A$, $B$ on $\left[0, \infty\right)$  by 
\begin{equation}\label{IFD}
A(s)=
\begin{cases}
-s^{2}\,\mbox{Log}(s^{2}), &\text{if $0\leq s\leq e^{-3}$;}\\
3s^{2}+4e^{-3}s^{}-e^{-6}, &\text{if $ s\geq e^{-3}$;}
\end{cases}
\quad  B(s)=F(s)+A(s).
\end{equation}
Furthermore, let  $a$, $b$ be functions defined by
\begin{equation}\label{abapex}
a(z)=\frac{z}{|z|^{2}}\,A(\left|z\right|)\,\, \text{ and  }\,\, b(z)=\frac{z}{|z|^{2}}\,B(\left|z\right|) \text{  for $z\in \mathbb{C}$, $z\neq 0$}.
\end{equation}
Notice that we have $b(z)-a(z)=z\,\mathrm{Log}\left|z\right|^{2}$. It follows that $A$ is a nonnegative  convex and increasing function, and $A\in C^{1}\left([0,+\infty)\right)\cap C^{2}\left((0,+\infty)\right)$. The Orlicz space $L^{A}(\mathbb{R}^{N})$ corresponding to $A$ is defined by
\begin{equation*}
L^{A}(\mathbb{R}^{N})=\left\{u\in L^{1}_{\text{loc}}(\mathbb{R}^{N}) : A(\left|u\right|)\in L^{1}_{}(\mathbb{R}^{N})\right\}, 
\end{equation*}
equipped with the Luxemburg norm 
\begin{equation*}
\left\|u\right\|_{L^{A}}={\inf}\left\{k>0: \int_{\mathbb{R}^{N}}A\left(k^{-1}{\left|u(x)\right|}\right)dx\leq 1 \right\}.
\end{equation*}
Here as usual $L^{1}_{\text{loc}}(\mathbb{R}^{N})$ is the space of all locally Lebesgue integrable functions. It is proved in \cite[Lemma 2.1]{CL} that $A$ is a Young-function which is $\Delta_{2}$-regular and $\left(L^{A}(\mathbb{R}^{N}),\|\cdot\|_{L^{A}} \right)$ is a separable reflexive  Banach space.

We consider the reflexive Banach space $W^{s}({\mathbb{R}}^{N})=H^{s}(\mathbb{R}^{N})\cap L^{A}(\mathbb{R}^{N})$ equipped with the usual norm $\left\|u\right\|_{W^{s}(\mathbb{R}^{N})}=\left\|u\right\|_{H^{s}(\mathbb{R}^{N})}+\left\|u\right\|_{L^{A}}$. The following lemma provides an alternative way of defining the energy space $W^{s}({\mathbb{R}}^{N})$.
\begin{lemma}\label{L11}
Let $0<s<1$ and $N\geq2$. Then 
\begin{equation*}
W^{s}(\mathbb{R}^{N})=\left\{u\in H^{s}(\mathbb{R}^{N}):\left|u\right|^{2}\mathrm{Log}\left|u\right|^{2}\in L^{1}(\mathbb{R}^{N})\right\}
\end{equation*}
\end{lemma}
\begin{proof} One easily verifies that  for every  $\epsilon>0$, there exist $C_{\epsilon}>0$ such that $|B(z)-B(z_{1})|\leq C_{\epsilon}(|z|^{1+\epsilon}+|z_{1}|^{1+\epsilon})|z-z_{1}|$ for all $z$, $z_{1}\in \mathbb{C}$.  Integrating this inequality on $\mathbb{R}^{N}$ with $\epsilon=(2^{*}_{s}-2)/2$ and applying Hölder inequality and Sobolev embeddings give
\begin{equation}\label{DB}
\int_{\mathbb{R}^{N}}\left|B(\left|u\right|)- B(\left|v\right|)\right|dx\leq C\left(\left\|u\right\|^{}_{H^{s}(\mathbb{R}^{N})}+ \left\|v\right\|^{}_{H^{s}(\mathbb{R}^{N})} \right)^{\gamma}\left\|u-v\right\|_{{L^{2}}},
\end{equation}
with $\gamma={2^{*}_{s}}/{2}$. Thus for $u\in H^{s}(\mathbb{R}^{N})$ we get  $B(|u|)\in L^{1}(\mathbb{R}^{N})$. Lemma \ref{L11}  follows then from the definition of the spaces  $W^{s}(\mathbb{R}^{N})$ and $L^{A}(\mathbb{R}^{N})$.
\end{proof} 

The following lemma is a variant of the Br\'ezis-Lieb lemma from \cite{LBL}.

\begin{lemma} \label{L4}
Let  $\left\{u_{n}\right\}$ be a bounded sequence in $W^{s}(\mathbb{R}^{N})$ such that $u_{n}\rightarrow u$ a.e. in $\mathbb{R}^{N}$. Then $u\in W^{s}(\mathbb{R}^{N})$ and 
\begin{equation*}
\lim_{n\rightarrow \infty}\int_{\mathbb{R}^{N}}\left\{\left|u_{n}\right|^{2}\mathrm{Log}\left|u_{n}\right|^{2}-\left|u_{n}-u\right|^{2}\mathrm{Log}\left|u_{n}-u\right|^{2}\right\}dx=\int_{\mathbb{R}^{N}}\left|u\right|^{2}\mathrm{Log}\left|u\right|^{2}dx.
\end{equation*}
\end{lemma}
\begin{proof}
The proof follows along the same lines as \cite[Lemma 2.3]{AHA1}. We omit the details.
\end{proof}

It follows from Proposition 1.1.3 in \cite{CB} that
\begin{equation*}
{W}^{-s}({\mathbb{R}^{N}})=H^{-s}(\mathbb{R}^{N})+L^{A^{\prime}}(\mathbb{R}^{N}),
\end{equation*}
where the Banach space ${W}^{-s}({\mathbb{R}^{N}})$ is equipped with its usual norm. Here, $L^{A^{\prime}}(\mathbb{R}^{N})$ is the dual space of $L^{A}(\mathbb{R}^{N})$ (see \cite{CL}). It is easy to see that one has the following chain of continuous embedding: $W^{s}(\mathbb{R}^{N})\hookrightarrow L^{2}(\mathbb{R}^{N})\hookrightarrow {W}^{-s}({\mathbb{R}^{N}})$.

Now we will show that the energy functional $E$ is of class $C^{1}$ on ${W}^{s}({\mathbb{R}^{N}})$. First we need the following lemma.

\begin{lemma} \label{APEX23}
The operator $L: u\rightarrow (-\Delta)^{s}u-u\,  \mathrm{Log}\left|u\right|^{2}$ is continuous from  $W^{s}(\mathbb{R}^{N})$  to $W^{-s}(\mathbb{R}^{N})$. The image under $L$ of a bounded subset of $W^{s}(\mathbb{R}^{N})$ is a bounded subset of $W^{-s}(\mathbb{R}^{N})$.
\end{lemma}
\begin{proof}
First, notice that $(-\Delta)^{s}$ is continuous from  $H^{s}(\mathbb{R}^{N})$ to $H^{-s}(\mathbb{R}^{N})$. Thus, using ${W}^{s}({\mathbb{R}}^{N})\hookrightarrow H^{s}(\mathbb{R}^{N})$, we obtain that the  operator $u\rightarrow (-\Delta)^{s}u$ is continuous from ${{W}^{s}}({\mathbb{R}^{N}})$ to ${W}^{-s}({\mathbb{R}^{N}})$. Secondly, for every  $\epsilon>0$, there exist $C_{\epsilon}>0$ such that $|b(z)-b(z_{1})|\leq C_{\epsilon}(|z|^{\epsilon}+|z_{1}|^{\epsilon})|z-z_{1}|$ for all $z$, $z_{1}\in \mathbb{C}$. Integrating this inequality on $\mathbb{R}^{N}$ with $\epsilon=(2^{*}_{s}-2)/2$ and applying Hölder inequality and Sobolev embeddings we obtain 
\begin{equation*}
\|b(u)-b(v)\|_{L^{2}(\mathbb{R}^{N})}\leq C \|u-v\|_{H^{s}(\mathbb{R}^{N})}\left(\|u\|_{H^{s}(\mathbb{R}^{N})}+\|u\|_{H^{s}(\mathbb{R}^{N})}\right)^{\gamma}
\end{equation*}
where $\gamma=2s/(N-2s)$. Then clearly $u\rightarrow b(u)$ is continuous and bounded from $H^{s}(\mathbb{R}^{N})$ to $L^{2}(\mathbb{R}^{N})$, then from ${{W}^{s}}({\mathbb{R}^{N}})$ to ${W}^{-s}({\mathbb{R}^{N}})$. Finally, since $u\rightarrow a(u)$ is continuous and bounded  from $L^{A}({\mathbb{R}}^{N})$ to $L^{A^{\prime}}({\mathbb{R}}^{N})$ (see  \cite[Lemma 2.6]{CL}), it follows that the operator $u\rightarrow a(u)-b(u)=-u\,\mathrm{Log}\left|u\right|^{2}$  is continuous and bounded from ${{W}^{s}}({\mathbb{R}^{N}})$ to ${W}^{-s}({\mathbb{R}^{N}})$, and lemma is proved.
\end{proof}
From Lemma \ref{APEX23},  we have the following consequence:

\begin{proposition} \label{DFFE}
The operator $E: W^{s}(\mathbb{R}^{N})\rightarrow \mathbb R$  is of class $C^{1}$ and  for $u\in W^{s}(\mathbb{R}^{N})$ the  Fr\'echet derivative of $E$ in $u$ exists and it is given by  
\begin{equation*}
E^{\prime}(u)=(-\Delta)^{s}u-u\, \mathrm{Log}\left|u\right|^{2}-u 
\end{equation*}
\end{proposition}
\begin{proof}
We first show that $E$  is continuous. Notice that 
\begin{equation}\label{CCC}
E(u)=\frac{1}{2}\int_{\mathbb{R}^{N}}|(-\Delta)^\frac{s}{2}u|^{2}dx+\frac{1}{2}\int_{\mathbb{R}^{N}}A(\left|u_{}\right|)dx-\frac{1}{2}\int_{\mathbb{R}^{N}}B(\left|u_{}\right|)dx.
\end{equation}
The first term in the right-hand side of \eqref{CCC} is continuous of $H^{s}(\mathbb{R}^{N})\rightarrow \mathbb R$, and it follows from 
Proposition \ref{orlicz}(i) in Appendix that the second term is continuous of $L^{A}(\mathbb{R}^{N})\rightarrow \mathbb{R}$. Moreover, by \eqref{DB} we get that the third term  in the right-hand side of \eqref{CCC} is continuous of $H^{s}(\mathbb{R}^{N})\rightarrow \mathbb R$. Therefore, $E\in C(W^{s}(\mathbb{R}^{N}),\mathbb{R})$. Now, direct calculations show that, for $u$, $v\in W^{s}(\mathbb{R}^{N})$, $t\in (-1,1)$ (see \cite[Proposition 2.7]{CL}),
\begin{equation*}
\lim_{t\rightarrow 0} \frac{E(u+tv)-E(u)}{t}=\bigl\langle (-\Delta)^{s}u-u\, \mbox{Log}\left|u\right|^{2}-u,v\bigl\rangle.
\end{equation*}
Thus, $E$ is G\^ateaux differentiable. Then, by Lemma \ref{APEX23} we see that $E$ is  Fr\'echet differentiable  and $E^{\prime}(u)=(-\Delta)^{s}u-u\, \mbox{Log}\left|u\right|^{2}-u$.
\end{proof}

\section{The Cauchy problem}
\label{S:2}
In this section we sketch the proof of the global well-posedness of the Cauchy Problem  for \eqref{0NL} in the energy space  ${W}^{s}(\mathbb{R}^{N})$. The proof of Proposition \ref{PCS} is an adaptation of the proof of \cite[Theorem 9.3.4]{CB}.  So, we will approximate the logarithmic nonlinearity by a smooth nonlinearity, and as a consequence we construct a sequence of global solutions of the regularized Cauchy problem in $C(\mathbb{R},H^{s}(\mathbb{R}^{N}))\cap C^{1}(\mathbb{R},H^{-s}(\mathbb{R}^{N}))$,  then we pass to the limit using standard compactness results, extract a subsequence which converges to the solution of the limiting equation \eqref{0NL}. Finally, by using special properties  of the logarithmic nonlinearity we establish  uniqueness of the global solution.

First we regularize the logarithmic nonlinearity near the origin.  For $z\in \mathbb{C}$ and $m\in \mathbb{N}$,  we define the functions $a_{m}$ and $b_{m}$ by 
\begin{equation*}
a_{m}(z)=
\begin{cases}
 a_{}(z), &\text{if $\left|z\right|\geq \frac{1}{m}$;}\\
m\,z\,a_{}(\frac{1}{m}) , &\text{if $\left|z\right|\leq \frac{1}{m}$;}
\end{cases}
\quad \text{and} \quad 
b_{m}(z)=
\begin{cases}
 b_{}(z) , &\text{if $\left|z\right|\leq {m}$;}\\
\frac{z}{m}\,b({m}) , &\text{if $\left|z\right|\geq {m}$,}
\end{cases}
\end{equation*}
where   $a$ and $b$ were defined in \eqref{abapex}. For any fixed $m\in \mathbb{N}$,  we define a family of regularized nonlinearities in the form  $g_{m}(z)=b_{m}(z)-a_{m}(z)$,  for every $z\in \mathbb{C}$.

In order to construct a solution of  \eqref{0NL}, we solve first, for $m\in \mathbb{N}$, the regularized Cauchy problem
\begin{equation}\label{AHAX}
i\partial_{t}u^{m}-(-\Delta)^{s} u^{m}+g_{m}(u^{m})=0.
\end{equation}
\begin{proposition} \label{APCS} Let $0<s<1$ and $N\geq2$. For any $u_{0}\in H^{s}(\mathbb{R}^{N})$, there is a unique  solution  $u^{m}\in C(\mathbb{R},H^{s}(\mathbb{R}^{N}))\cap C^{1}(\mathbb{R}, H^{-s}(\mathbb{R}^{N}))$ of \eqref{AHAX}  such that $u^{m}(0)=u_{0}$. Furthermore, the conservation of energy and charge holds; that is, 
\begin{equation*}
\mathcal{E}_{m}(u^{m}(t))=\mathcal{E}_{m}(u_{0})\quad  and \quad \left\|u^{m}(t)\right\|^{2}_{L^{2}}=\left\|u_{0}\right\|^{2}_{L^{2}}\quad  \text{for all $t\in \mathbb{R}$},
\end{equation*}
where 
\begin{equation*}
\mathcal{E}_{m}(u)=\frac{1}{2}\int_{\mathbb{R}^{N}}|(-\Delta)^\frac{s}{2} u|^{2}dx-\frac{1}{2}\int_{\mathbb{R}^{N}}G_{m}(u)dx, \quad G_{m}(z)=\int^{\left|z\right|}_{0}g_{m}(s)ds.
\end{equation*}
\end{proposition}
\begin{proof}
Our proof is inspired by the results of \cite[Section 4]{CHHHHW}. First, since $g_{m}$ is globally Lipschitz continuous $\mathbb{C}\rightarrow \mathbb{C}$, one easily verifies that $\|g_{m}(u)-g_{m}(v)\|_{H^{-s}(\mathbb{R}^{N})}\leq C(K)\|u-v\|_{H^{s}(\mathbb{R}^{N})}$ provided that $\|u\|_{H^{s}(\mathbb{R}^{N})}+\|v\|_{H^{s}(\mathbb{R}^{N})}\leq K$. Then, from \cite[Proposition 4.1]{CHHHHW} we see that there exists  a weak solution $u^{m}$ of \eqref{AHAX}  such that 
\begin{align}
& u^{m}\in L^{\infty}((-T_{min},T_{max}), H^{s}(\mathbb{R}^{N}))\cap W^{1,\infty}((-T_{min},T_{max}), H^{-s}(\mathbb{R}^{N})) \nonumber  ,\\
& \mathcal{E}_{m}(u^{m}(t))\leq\mathcal{E}_{m}(u_{0})\quad  \text{and} \quad \left\|u^{m}(t)\right\|^{2}_{L^{2}}=\left\|u_{0}\right\|^{2}_{L^{2}}\label{LON}
\end{align}
for all $t\in (-T_{min},T_{max})$, where $(-T_{min},T_{max})$ is the maximal  existence time interval of $u^{m}$ for initial data $u_{0}$.

Secondly, we show that  there is uniqueness for the problem \eqref{AHAX}. In fact, let $I$ be an interval containing $0$ and let $\Phi$, $\Psi\in L^{\infty}(I, H^{s}(\mathbb{R}^{N}))\cap W^{1,\infty}(I,H^{-s}(\mathbb{R}^{N}))$ be two solutions of \eqref{AHAX}. It follows that 
\begin{equation*}
\Psi(t)-\Phi(t)=i\,\int^{t}_{0}U(t-s)\left(g_{m}(\Psi(s))-g_{m}(\Phi(s))\right)ds \quad \text{for all} \quad t\in I,
\end{equation*}
where $U(t)=e^{-it(-\Delta)^{s}}$. Since $g_{m}$ is  Lipschitz continuous  $L^2(\mathbb{R}^{N})\rightarrow L^2(\mathbb{R}^{N})$, there exists a constant  $C>0$ such that 
\begin{equation*}
\|\Psi(t)-\Phi(t)\|^{2}_{L^{2}}\leq C \int^{t}_{0}\|\Psi(s)-\Phi(s)\|^{2}_{L^{2}}ds,
\end{equation*}
and therefore the uniqueness  follows by Gronwall's Lemma. Furthermore, since $\int_{\mathbb{R}^{N}}G_{m}(u)dx\leq C\left\|u\right\|^{2}_{L^{2}}$, from \eqref{LON}  we get  $\|u^{m}(t)\|^{2}_{H^{s}(\mathbb{R}^{N})}\leq C\left\|u_{0}\right\|^{2}_{L^{2}}+4\mathcal{E}_{m}(u_{0})$ for all $t\in (-T_{min},T_{max})$.   The continuity argument implies that all solutions of \eqref{AHAX} are global and uniformly bounded in $H^{s}(\mathbb{R}^{N})$.

Finally, we prove that the weak solution $u^{m}$ of  \eqref{AHAX} satisfies the conservation of energy. Indeed, fix $t_{0}\in \mathbb{R}$. Let $\varphi_0=u^{m}(t_{0})$ and let $w$ be the solution of \eqref{AHAX} with $w(0)=\varphi_0$. By uniqueness, we see that $w(\cdot-t_{0})=u^{m}(\cdot)$ on $\mathbb{R}$. From \eqref{LON}, we deduce in particular that  
\begin{equation*}\mathcal{E}_{m}(u_{0})\leq \mathcal{E}_{m}(\varphi_0)=\mathcal{E}_{m}(u^{m}(t_{0}))
\end{equation*}
Therefore, we have that both $\left\|u^{m}(t)\right\|^{2}_{L^{2}}$ and $\mathcal{E}_{m}(u^{m}(t))$ are constant on $\mathbb{R}$. The inclusion $u^{m}\in C(\mathbb{R},H^{s}(\mathbb{R}^{N}))\cap C^{1}(\mathbb{R}, H^{-s}(\mathbb{R}^{N}))$ follows from conservation laws. This completes the proof of Proposition \ref{APCS}.
\end{proof}

For the proof of Proposition \ref{PCS}, we will use the following lemma.
\begin{lemma}\label{3ACS} Let $0<s<1$ and $N\geq2$. Given $k\in\mathbb{N}$, set $\Omega_{k}=\left\{x\in \mathbb{R}^{N}: |x|<k\right\}$.
Let $\left\{u^{{m}}\right\}_{m\in\mathbb{N}}$ be a bounded sequence in $L^{\infty}(\mathbb{R},H^{s}_{}(\mathbb{R}^{N}))$. If $(u^{m}|_{\Omega_{k}})_{m\in\mathbb{N}}$ is a  bounded sequence of  $W^{1,\infty}(\mathbb{R}, H^{-s}(\Omega_{k}))$ for $k\in\mathbb{N}$, then there exists a subsequence, which we still denote by $\left\{u^{{m}}\right\}_{m\in\mathbb{N}}$, and there exist  $u\in L^{\infty}(\mathbb{R},H^{s}(\mathbb{R}^{N}))$, such that the following properties hold:\\
{\rm (i)}$\left.u\right|_{\Omega_{k}}\in W^{1,\infty}(\mathbb{R},H^{-s}(\Omega_{k}))$ for every $k\in\mathbb{N}$.\\
{\rm (ii)}  $u^{m}(t)\rightharpoonup u^{}(t)$  in  $H^{s}(\mathbb{R}^{N})$  as $m\rightarrow \infty$ for every $t\in \mathbb{R}$.\\
{\rm (iii)}	For every $t\in\mathbb{R}$ there exists a subsequence  $m_{j}$ such that $u_{}^{m_{j}}(x,t)\rightarrow u_{}^{}(x,t)$   as $j\rightarrow \infty$, for a.e. $x\in \mathbb{R}^{N}$.\\
{\rm (iv)}	$u_{}^{m_{}}(x,t)\rightarrow u_{}^{}(x,t)$ as $m\rightarrow \infty$,  for a.e.  $(x,t)\in \mathbb{R}^{N}\times\mathbb{R}$. 
\end{lemma}
\begin{proof}
The proof follows a similar argument as in  Lemma 9.3.6 of \cite{CB} and we do not repeat here.
\end{proof}

\begin{proof}[ \bf {Proof of Proposition \ref{PCS}}] Our proof follows the ideas of Cazenave \cite[Theorem 9.3.4]{CB}. 
Applying Proposition \ref{APCS}, we see that  for every $m\in \mathbb{N}$ there exists a unique global solution $u^{m}\in C(\mathbb{R}, H^{s}(\mathbb{R}^{N}))\cap C^{1}(\mathbb{R},  H^{-s}(\mathbb{R}^{N}))$ of \eqref{AHAX}, which satisfies 
\begin{equation}\label{JKL}
\mathcal{E}_{m}(u^{m}(t))=\mathcal{E}_{m}(u_{0})\quad \mbox{and}\quad\left\|u^{m}(t)\right\|^{2}_{L^{2}}=\left\|u_{0}\right\|^{2}_{L^{2}} \quad \text{ for all $t\in \mathbb{R}$},
\end{equation}
where \begin{equation*}
\mathcal{E}_{m}(u)=\frac{1}{2}\int_{\mathbb{R}^{N}}|(-\Delta)^\frac{s}{2} u|^{2}dx +\frac{1}{2}\int_{\mathbb{R}^{N}}\Phi_{m}(\left|u_{}\right|)dx-\frac{1}{2}\int_{\mathbb{R}^{N}}\Psi_{m}(\left|u_{}\right|)dx,
\end{equation*}
and the functions $\Phi_{m}$ and $\Psi_{m}$  defined by
\begin{equation*}
\Phi_{m}(z)=\frac{1}{2}\int^{\left|z\right|}_{0}a_{m}(s)ds \quad \mbox{and} \quad \Psi_{m}(z)=\frac{1}{2}\int^{\left|z\right|}_{0}b_{m}(s)ds.
\end{equation*}
It follows from \eqref{JKL} that $u^{m}$ is bounded in $L^{\infty}(\mathbb{R}, L^{2}(\mathbb{R}^{N}))$. Moreover, we have that the sequence of approximating solutions $u^{m}$ is bounded in the space $L^{\infty}(\mathbb{R}, H^{s}(\mathbb{R}^{N}))$. It also follows from the NLS equation \eqref{AHAX} that the sequence $\left.\partial_{t}u^{m}\right|_{\Omega_{k}}$ is bounded in the space $L^{\infty}(\mathbb{R}, H^{-s}(\Omega_{k}))$, where $\Omega_{k}=\left\{x\in \mathbb{R}^{N}: |x|<k\right\}$. Such statements can be proved along the same lines as the Step 2 of Theorem 9.3.4 in \cite{CB}. Therefore, we have that  $\left\{u^{{m}}\right\}_{m\in\mathbb{N}}$ satisfies the assumptions of  Lemma \ref{3ACS}. Let $u$ be the limit of $u^{m}$.

Now we show that the limiting function $u\in L^{\infty}(\mathbb{R},H^{s}(\mathbb{R}^{N}))$ is a weak solution of the logarithmic NLS equation \eqref{0NL}. To do so, we first write  a weak formulation of the NLS equation \eqref{AHAX}. Indeed, for any test  function $\psi\in C^{\infty}_{0}({\mathbb{R}^{N}})$  and $\phi\in C^{\infty}_{0}({\mathbb{R}})$, we have
\begin{equation}\label{3DPL}
\int_{\mathbb{R}}\left[-\left\langle i\, u^{m}, \psi\right\rangle \phi^{\prime}(t)- \left\langle u^{m}, (-\Delta)^{s}\psi\right\rangle \phi^{}(t)\right]\,dt+\int_{\mathbb{R}}\int_{\mathbb{R}^{N}}g_{m}(u^{m})\psi\phi\,dx\,dt=0.
\end{equation}
Furthermore, since $g_{m}(z)\rightarrow z\, \mbox{Log}\left|z\right|^{2}$ pointwise in $z\in \mathbb{C}$ as $m\rightarrow+\infty$, we apply the properties (ii)-(iv) of Lemma \ref{3ACS} to the integral formulation \eqref{3DPL} and obtain the following integral equation (see proof of Step 3 of \cite[Theorem 9.3.4]{CB})
 \begin{equation}\label{ert}
\int_{\mathbb{R}}\left[-\left\langle i\, u^{}, \psi\right\rangle \phi^{\prime}(t)- \left\langle u^{}, (-\Delta)^{s}\psi\right\rangle \phi^{}(t)\right]\,dt+\int_{\mathbb{R}}\int_{\mathbb{R}^{N}}u_{}\, \mbox{Log}\left|u_{}\right|^{2}\psi\phi\,dx\,dt=0.
\end{equation}
In addition, $u(0)=u_{0}$ by property (ii) of Lemma \ref{3ACS}. Moreover,  it is easy to see that $u\in{L^{\infty}(\mathbb{R}, L^{A}(\mathbb{R}^{N}) )}$ . Therefore, by integral equation \eqref{ert},  $u\in{L^{\infty}(\mathbb{R}, W^{s}(\mathbb{R}^{N}))}$ is a weak solution of the logarithmic NLS equation \eqref{0NL}. In particular, from Lemma \ref{APEX23}, we deduce that $u\in W^{1,\infty}(\mathbb{R}, W^{-s}(\mathbb{R}^{N}))$.  Now we show uniqueness of the solution in the class $L^{\infty}(\mathbb{R}, W^{s}(\mathbb{R}^{N}))\cap W^{1,\infty}(\mathbb{R}, W^{-s}(\mathbb{R}^{N}))$. Indeed, let $u$ and $v$ be two solutions   of \eqref{0NL} in that class. On taking the difference of the two equations and taking the $W^{s}(\mathbb{R}^{N})-W^{-s}(\mathbb{R}^{N})$ duality product with $i(u-u)$, we see that
\begin{equation*}
\left\langle u_{t}-v_{t}, u-v\right\rangle=-\Im \int_{\mathbb{R}^{N}}\left(u\mathrm{Log}\left|u \right|^{2}-v\mathrm{Log}\left|v \right|^{2} \right)(\overline{u}-\overline{v})dx.
\end{equation*}
Thus, from \cite[Lemma 9.3.5]{CB} we obtain
\begin{equation*}
\left\|u(t)-v(t)\right\|^{2}_{L^{2}}\leq 8\int^{t}_{0}\left\|u(s)-v(s)\right\|^{2}_{L^{2}}ds.
\end{equation*}
Therefore, the uniqueness of a solution follows by Gronwall's Lemma.  Finally, the conservation of charge and energy, and the continuity of the solution $u\in C(\mathbb{R}, W^{s}(\mathbb{R}^{N}))\cap C^{1}(\mathbb{R}, W^{-s}(\mathbb{R}^{N}))$ in time $t$ follow from the arguments identical to the case of the classical logarithmic NLS equation (see proof of Step 4 of \cite[Theorem 9.3.4]{CB}). This finishes the proposition.
\end{proof}

\section{Variational Analysis}
\label{S:3}
The aim of this section is to prove Theorem \ref{ESSW}. First we recall the fractional logarithmic Sobolev inequality. For a proof we refer to \cite{ACNT}.
\begin{lemma} \label{L1}
Let $f$ be any function in $ H^{s}(\mathbb{R}^{N})$ and $\alpha>0$. Then
\begin{equation}\label{SII}
\int_{\mathbb{R}^{N}}\left|f(x)\right|^{2}\mathrm{Log}\left(\frac{\left|f(x)\right|^{2}}{\|f\|^{2}_{L^{2}}}\right)dx+\left(N+\frac{N}{s}\mathrm{Log}\,\alpha+\mathrm{Log}\frac{s\Gamma(\frac{N}{2})}{\Gamma(\frac{N}{2s})}\right)\|f\|^{2}_{L^{2}}\leq\frac{\alpha^{2}}{\pi^{s}}\|(-\Delta)^{\frac{s}{2}}f\|^{2}_{L^{2}}.
\end{equation}
\end{lemma}
\begin{lemma}\label{L2}
Let $\omega\in \mathbb{R}$. Then, the quantity $d(\omega)$ is positive and satisfies
\begin{equation}\label{EA}
d(\omega)\geq \frac{1}{2}\left(\frac{s\Gamma(\frac{N}{2})}{\Gamma(\frac{N}{2s})}\right){\pi}^{\frac{N}{2}}e^{\omega+N} .
\end{equation}
\end{lemma}
\begin{proof}
Let $u\in W^{s}(\mathbb{R}^{N})  \setminus  \left\{0 \right\}$ be such that  $I_{\omega}(u)=0$. Using the fractional logarithmic Sobolev inequality with $\alpha=\pi^{\frac{s}{2}}$, we see that
\begin{equation*}
\left(\omega+N(1+\mbox{Log}(\sqrt{\pi}))+\mathrm{Log}\frac{s\Gamma(\frac{N}{2})}{\Gamma(\frac{N}{2s})}\right)\left\|u\right\|^{2}_{L^{2}}\leq \left(\mbox{Log}\left\|u\right\|^{2}_{L^{2}}\right)\left\|u\right\|^{2}_{L^{2}}.
\end{equation*}
Thus, by the definition of $d(\omega)$ given in \eqref{MPE}, we get \eqref{EA}.
\end{proof}

\begin{lemma} \label{L5}
Let  $\omega\in \mathbb{R}$ and $0<s<1$. If  $\left\{ u_{n}\right\}$  is a minimizing sequence of problem \eqref{MPE}, then there is a subsequence   $\left\{u_{j_{n}}\right\}$ and a sequence $\left\{y_{n}\right\}\subset \mathbb{R}^{N}$ such that 
\begin{equation*}
v_{n}(x):=u_{j_{n}}(x+y_{n})
\end{equation*}
converges weakly in $W^{s}(\mathbb{R}^{N})$ to a function $\varphi\neq0$. Moreover, $\left\{v_{n}\right\}$ converges to $\varphi$ a.e and in $L_{\text{loc}}^{q}(\mathbb{R}^{N})$ for every $q\in[2,2^{*}_{s})$.
\end{lemma}
\begin{proof}
Let $\left\{ u_{n}\right\} \subseteq W^{s}(\mathbb{R}^{N})$ be a minimizing sequence for $d(\omega)$, then the sequence $\left\{ u_{n}\right\}$ is bounded in  $W^{s}(\mathbb{R}^{N})$. Indeed, it is clear that the sequence $\|u_{n}\|^{2}_{L^{2}}$ is bounded. Moreover, using the fractional logarithmic Sobolev inequality and recalling that $I_{\omega}(u_{n})=0$, we obtain
\begin{equation*}
\left(1-\frac{\alpha^{2}}{\pi^{s}}\right)\left\|(-\Delta)^{\frac{s}{2}}u_{n}\right\|^{2}_{L^{2}}\leq \mbox{Log}\left[\left(\frac{e^{-\left(\omega+N\right)}\Gamma\left(N/2s\right)}{s\,\alpha^{N/s}\Gamma\left(N/2\right)}\right)\left\|u_{n}\right\|^{2}_{L^{2}}\right]\left\|u_{n}\right\|^{2}_{L^{2}}.
\end{equation*}
Taking $\alpha>0$ sufficiently small, we see that $\|(-\Delta)^{\frac{s}{2}}u_{n}\|^{2}_{L^{2}}$ is bounded, so the sequence $\left\{ u_{n}\right\}$ is bounded in $H^{s}(\mathbb{R}^{N})$.  Then, using  $I_{\omega}(u_{n})=0$ again, and \eqref{DB} we obtain
\begin{equation*}
\int_{\mathbb{R}^{N}}A\left(\left|u_{n}\right|\right)dx\leq \int_{\mathbb{R}^{N}}B\left(\left|u_{n}\right|\right)dx+|\omega|\left\|u_{n}\right\|^{2}_{L^{2}}\leq C,
\end{equation*}
which implies, by \eqref{DA1} in the Appendix, that the sequence $\left\{ u_{n}\right\}$ is bounded in $W^{s}(\mathbb{R}^{N})$. Furthermore, since $W^{s}(\mathbb{R}^{N})$ is a reflexive Banach space, there is $v \in W^{s}(\mathbb{R}^{N})$ such that, up to a subsequence, $u_{n}\rightharpoonup v$ weakly in $W^{s}(\mathbb{R}^{N})$.

On the other hand, let $2<q<2^{\ast}_{s}$ and $0<\delta<1$.  Notice that  $\int_{\mathbb{R}^{N}}|u_{n}|^{2}\mbox{Log}\left|u_{n}\right|^{2}dx=\|(-\Delta)^{\frac{s}{2}}u_{n}\|^{2}_{L^{2}}+\omega\|u_{n}\|^{2}_{L^{2}}\geq -M(\omega)$ for sufficiently large $n$, where $M(\omega)$ is a positive constant depending only on $\omega$.  Arguing as in the proof of Lemma 3.3  in \cite{CL}, it is easy to see that
\begin{equation*}
2d(\omega)\leq (\delta^{-(q-2)}-\left[(q-2)\mbox{Log}\delta^{2}\right]^{-1})\int_{\mathbb{R}^{N}}\left|u_{n}\right|^{q}dx-\left[\mbox{Log}\delta^{2}\right]^{-1}M(\omega).
\end{equation*}
We now choose $\delta$ such that $-\left[\mbox{Log}\delta^{2}\right]^{-1}M(\omega)=d(\omega)\left((q-2)/(q+2)\right)$. Easy computations permit us to obtain 
\begin{equation}\label{CPC}
\int_{\mathbb{R}^{N}}\left|u_{n}\right|^{q}dx\left(e^{M(\omega)(q+2)/2d(\omega)}+\frac{d(\omega)}{M(\omega)(q+2)}\right)\geq 2d(\omega)\left(\frac{q+6}{q+2}\right)\geq d(\omega).
\end{equation}
Combining \eqref{CPC} and Lemma \ref{CLLQ} implies 
\begin{equation*}
\sup_{y\in\mathbb{R}^{N}} \int _{B_{1}(y)}\left|u_{n}\right|^{2}dx\geq\epsilon>0,
\end{equation*}
In this case we can choose $\left\{y_{n}\right\}\subset \mathbb{R}^{N}$ such that
\begin{equation*}
\int _{B_{1}(0)}\left|u_{n}(\cdot+y_{n})\right|^{2}dx\geq\epsilon^{\prime},
\end{equation*}
where $0<\epsilon^{\prime}<\epsilon$, and hence, due to the compactness of the embedding  $H_{\text{loc}}^{s}(\mathbb{R}^{N})\hookrightarrow L_{\text{loc}}^{2}(\mathbb{R}^{N})$, we deduce that the translated sequence $v_{n}:=u_{n}(\cdot+y_{n})$ has a weak limit $\varphi$ in $H^{s}(\mathbb{R}^{N})$ that is not identically zero. Also,  it follows that $\left\{v_{n}\right\}$ converges to $\varphi$ strongly in $L_{\text{loc}}^{q}(\mathbb{R}^{N})$ for any $q\in[2,2^{*}_{s})$. Therefore,  we infer that, after a translation if necessary, $\left\{u_{n}\right\}$ converges weakly in $W^{s}(\mathbb{R}^{N})$ and a.e. to a function $\varphi\neq0$.  Hence the result is established.
\end{proof}

\begin{proof}[ {\bf{Proof of Theorem \ref{ESSW}}}] The proof follows basically the same idea as the proof of \cite[Proposition 1.3 and Lemma 3.1]{AHA1} (see also \cite{AnguloArdila2016}). Let $\left\{ u_{n}\right\} \subseteq W^{s}(\mathbb{R}^{N})$ be a minimizing sequence for $d(\omega)$. From Lemma \ref{L5}, there exist $\varphi\in W^{s}(\mathbb{R}^{N})\setminus\left\{0\right\}$  such that, $v_{n}:=u_{j_{n}}(\cdot+y_{n})\rightharpoonup \varphi$
 weakly in $W^{s}(\mathbb{R}^{N})$ and  $\left\{v_{n}\right\}$ converges to $\varphi$ a.e and in $L_{\text{loc}}^{q}(\mathbb{R}^{N})$ for every $q\in[2,2^{*}_{s})$.

Now we prove that $\varphi\in \mathcal{G}_{\omega}$, that is,   $I_{\omega}(\varphi)=0$ and $S_{\omega}(\varphi)=d(\omega)$. First,  assume by contradiction that $I_{\omega}(\varphi)<0$. By elementary computations, we can see that there is $0<\lambda<1$ such that $I_{\omega}(\lambda \varphi)=0$. Then, from the definition of $d(\omega)$ and  the weak lower semicontinuity of the $L^{2}(\mathbb{R}^{N})$-norm, we have
\begin{equation*}
d(\omega)\leq \frac{1}{2}\left\|\lambda \varphi\right\|^{2}_{L^{2}}<\frac{1}{2}\left\|\varphi\right\|^{2}_{L^{2}}\leq \frac{1}{2}\liminf\limits_{n\rightarrow \infty}\left\|v_{n}\right\|^{2}_{L^{2}}=d(\omega),
\end{equation*}
which is impossible. On the other hand, assume that $I_{\omega}(\varphi)>0$. Since the embedding  $W^{s}(\mathbb{R}^{N})\hookrightarrow {H^{s}(\mathbb{R}^{N})}$ is continuous, we see that $v_{n}\rightharpoonup \varphi$ weakly in $H^{s}(\mathbb{R}^{N})$. Thus, we have  
\begin{align}
& \left\|(-\Delta)^{\frac{s}{2}}v_{n}\right\|^{2}_{L^{2}}-\left\|(-\Delta)^{\frac{s}{2}}v_{n}-(-\Delta)^{\frac{s}{2}}\varphi\right\|^{2}_{L^{2}}-\left\|(-\Delta)^{\frac{s}{2}}\varphi\right\|^{2}_{L^{2}}\rightarrow0,  \label{2C11}\\
&\left\| v_{n}\right\|^{2}_{L^{2}}-\left\| v_{n}-\varphi\right\|^{2}_{L^{2}}-\left\| \varphi \right\|^{2}_{L^{2}}\rightarrow0\label{2C12} 
\end{align}
as $n\rightarrow\infty$. Combining \eqref{2C11}, \eqref{2C12} and Lemma \ref{L4} leads to 
\begin{equation*}
\lim_{n\rightarrow \infty}I_{\omega}(v_{n}-\varphi)=\lim_{n\rightarrow \infty}I_{\omega}(v_{n})-I_{\omega}(\varphi)=-I_{\omega}(\varphi),
\end{equation*}
which combined with  $I_{\omega}(\varphi)> 0$ give us  that $I_{\omega}(v_{n}-\varphi)<0$ for sufficiently large $n$. Thus, by \eqref{2C12} and  applying the same argument as above, we see that 
\begin{equation*}
d(\omega)\leq\frac{1}{2} \lim_{n\rightarrow \infty}\left\|v_{n}-\varphi\right\|^{2}_{L^{2}}=d(\omega)-\frac{1}{2}\left\|\varphi\right\|^{2}_{L^{2}},
\end{equation*}
which is a contradiction because $\|\varphi\|^{2}_{L^{2}}>0$. Then, we deduce that $I_{\omega}(\varphi)=0$. In addition, by the weak lower semicontinuity of the $L^{2}(\mathbb{R}^{N})$-norm,  we have
\begin{equation}\label{inequa}
d(\omega)\leq \frac{1}{2}\left\|\varphi\right\|^{2}_{L^{2}}\leq \frac{1}{2} \liminf\limits_{n\rightarrow \infty}\left\|v_{n}\right\|^{2}_{L^{2}}=d(\omega),
\end{equation}
which implies, by the definition of $d(\omega)$, that $\varphi\in \mathcal{G}_{\omega}$. 

Now we claim that $v_{n}\rightarrow \varphi$ strongly in $W^{s}(\mathbb{R}^{N})$. Indeed, by \eqref{2C12}, we infer that $v_{n}\rightarrow \varphi$  in $L^{2}(\mathbb{R}^{N})$. Moreover,  since the sequence $\left\{ v_{n}\right\}$ is bounded in $H^{s}(\mathbb{R}^{N})$, from \eqref{DB} we obtain
\begin{equation*}
 \lim_{n\rightarrow \infty}\int_{\mathbb{R}^{N}}B\left(\left|v_{n}(x)\right|\right)dx=\int_{\mathbb{R}^{N}}B\left(\left|\varphi(x)\right|\right)dx,\end{equation*}
which combined with $I_{\omega}(v_{n})=I_{\omega}(\varphi)=0$  for any $n\in \mathbb{N}$,  gives
\begin{equation}\label{2BX1}
\lim_{n\rightarrow \infty}\left[\|(-\Delta)^{\frac{s}{2}}v_{n}\|_{L^{2}}^{2}+\int_{\mathbb{R}^{N}}A\left(\left|v_{n}(x)\right|\right)dx\right]= \|(-\Delta)^{\frac{s}{2}}\varphi\|_{L^{2}}^{2}+\int_{\mathbb{R}^{N}}A\left(\left|\varphi(x)\right|\right)dx.
\end{equation}
Moreover, by \eqref{2BX1},  the weak lower semicontinuity of the $L^{2}(\mathbb{R}^{N})$-norm and Fatou lemma, we deduce (see e.g. \cite[Lemma 2.4.4]{TA})
\begin{align}
& \lim_{n\rightarrow \infty}\|(-\Delta)^{\frac{s}{2}} v_{n}\|_{L^{2}}^{2}=\|(-\Delta)^{\frac{s}{2}} \varphi\|_{L^{2}}^{2}\label{N1}\\
& \lim_{n \to \infty}\int_{\mathbb{R}^{N}}A\left(\left|v_{n}(x)\right|\right)dx=\int_{\mathbb{R}^{N}}A\left(\left|\varphi(x)\right|\right)dx. \label{N2} \end{align}
Since $v_{n}\rightharpoonup \varphi$ weakly  in $H^{s}(\mathbb{R}^{N})$, it follows from  \eqref{N1}  that $v_{n}\rightarrow\varphi$  in $H^{s}(\mathbb{R}^{N})$.  Finally, by Proposition  \ref{orlicz} {ii)} in Appendix  and \eqref{N2} we have $v_{n}\rightarrow\varphi$  in $L^{A}(\mathbb{R}^{N})$. Thus, by definition of the $W^{s}({\mathbb{R}^{N}})$-norm, we infer that $v_{n}\rightarrow\varphi$  in $W^{s}({\mathbb{R}}^{N})$ which concludes the proof.
\end{proof}

\section{Stability of the standing waves}
\label{S:4}

\begin{proof}[ {\bf{Proof of Theorem \ref{2ESSW}}}] 
We argue by contradiction.  Suppose that $\mathcal{G}_{\omega}$ is not $W^{s}(\mathbb{R}^{N})$-stable. Then there exist $\epsilon>0$, a sequence $(u_{n,0})_{n\in \mathbb{N}}$ in $W^{s}(\mathbb{R}^{N})$ such that
\begin{equation}\label{T21}
\inf_{\psi\in\mathcal{G}_{\omega}}\left\|u_{n,0}-\psi\right\|_{W^{s}(\mathbb{R}^{N})}<\frac{1}{n},
\end{equation}
and a sequence $(t_{n})_{n\in \mathbb{N}}$ such that
\begin{equation}\label{3C2}
\inf_{\psi\in\mathcal{G}_{\omega}} \|u_{n}(t_{n})-\psi\|_{W^{s}(\mathbb{R}^{N})}\geq{\epsilon},
\end{equation}
where $u_{n}$ denotes the solution of the Cauchy problem \eqref{0NL} with initial data $u_{n,0}$. Set $v_{n}(x)= u_{n}(x,t_{n})$. By \eqref{T21} and conservation laws, we obtain
\begin{gather}\label{CE1}
\left\|v_{n}\right\|^{2}_{L^{2}}=\left\|u_{n}(t_{n})\right\|^{2}_{L^{2}}=\left\|u_{n,0}\right\|^{2}_{L^{2}}\rightarrow 2d({\omega})\\
S_{\omega}(v_{n})=S_{\omega}(u_{n}(t_{n}))=S_{\omega}(u_{n,0})\rightarrow d(\omega),\label{A12}
\end{gather}
as $n\rightarrow \infty$.  Moreover, by combining \eqref{CE1} and \eqref{A12} lead us to $I_{\omega}(v_{n})\rightarrow 0$ as $n\rightarrow \infty$. Next, define the sequence $f_{n}(x)=\rho_{n}v_{n}(x)$ with
\begin{equation*}
\rho_{n}=\exp\left(\frac{I_{\omega}(v_{n})}{2\|v_{n}\|^{2}_{L^{2}}}\right),
\end{equation*}
where $\exp(x)$ represent the exponential function. It is clear that $\lim_{n\rightarrow \infty}\rho_{n}=1$ and $I_{\omega}(f_{n})=0$ for any $n\in\mathbb{N}$. Furthermore, since the sequence $\left\{v_{n}\right\}$  is bounded in $W^{s}(\mathbb{R}^{N})$, we get $\|v_{n}-f_{n}\|_{W^{s}(\mathbb{R}^{N})}\rightarrow 0$ as $n\rightarrow \infty$. Then, by  \eqref{A12}, we have that $\left\{f_{n}\right\}$ is a minimizing sequence for $d(\omega)$. Thus, by Theorem \ref{ESSW}, up to a subsequence, there exist $(y_{n})\subset \mathbb{R}^{N}$ and a function  $\varphi\in\mathcal{G}_{\omega}$ such that
\begin{equation}\label{UEP1}
\| f_{n}(\cdot-y_{n})- \varphi\|_{W^{s}(\mathbb{R}^{N})}\rightarrow 0 \quad \text{as $n\rightarrow +\infty$}.
\end{equation}
Since $\varphi(\cdot+y_{n})\in \mathcal{G}_{\omega}$, remembering that $v_{n}= u_{n}(t_{n})$ and substituting this in \eqref{UEP1}, we get 
\begin{equation*}
\inf_{\psi\in\mathcal{G}_{\omega}}\left\|u_{n}(t_{n})-\psi\right\|_{ W^{s}(\mathbb{R}^{N})}\leq \|v_{n}-f_{n}\|_{W^{s}(\mathbb{R}^{N})}+ \inf_{\psi\in\mathcal{G}_{\omega}}\left\|f_{n}-\psi\right\|_{ W^{s}(\mathbb{R}^{N})}\rightarrow 0\quad 
\end{equation*}
as $n\rightarrow +\infty$, which is a contradiction with \eqref{3C2}. This finishes the proof.
\end{proof}

\section*{Acknowledgements}
The author wishes to express his sincere thanks to the referees for their valuable comments.

\section{Appendix}
\label{S:5}

We list some properties of the Orlicz space $L^{A}(\mathbb{R}^{N})$ that we have used above.  For a proof of such statements we refer to \cite[Lemma 2.1]{CL}.

\begin{proposition} \label{orlicz}
Let $\left\{u_{{m}}\right\}$ be a sequence in  $L^{A}(\mathbb{R}^{N})$, the following facts hold:\\
{\it i)} If  $u_{{m}}\rightarrow u$ in $L^{A}(\mathbb{R}^{N})$, then $A(\left|u_{{m}}\right|)\rightarrow A(\left|u\right|)$ in $L^{1}(\mathbb{R}^{N})$ as   $n\rightarrow \infty$.\\
{\it ii)} Let  $u\in L^{A}(\mathbb{R}^{N})$. If  $u_{m}\rightarrow u$ $a.e.$ in $\mathbb{R}^{N}$ and if 
\begin{equation*}
\lim_{n \to \infty}\int_{\mathbb{R}^{N}}A\left(\left|u_{m}(x)\right|\right)dx=\int_{\mathbb{R}^{N}}A\left(\left|u(x)\right|\right)dx,
\end{equation*}
then $u_{{m}}\rightarrow u$ in $L^{A}(\mathbb{R}^{N})$ as   $n\rightarrow \infty$.\\
{\it iii)} For any $u\in L^{A}(\mathbb{R}^{N})$, we have
\begin{equation}\label{DA1}
{\rm min} \left\{\left\|u\right\|_{L^{A}},\left\|u\right\|^{2}_{L^{A}}\right\}\leq  \int_{\mathbb{R}^{N}}A\left(\left|u(x)\right|\right)dx\leq {\rm max} \left\{\left\|u\right\|_{L^{A}},\left\|u\right\|^{2}_{L^{A}}\right\}.
\end{equation}
\end{proposition}

\bibliographystyle{plain}
\bibliography{bibliografia}

\end{document}